\pdfoutput=1

\documentclass[final]{siamltex}

\usepackage{amsfonts}
\usepackage[leqno]{amsmath}
\usepackage{ amssymb }
\usepackage{tikz}
\usepackage{algorithmic}
\usepackage{url}
\usepackage{multirow}
\usepackage{colortbl}
\usepackage{booktabs}
\usepackage{siunitx}
\usepackage{bm}
\usepackage{mathtools}
\usepackage{caption}
\usepackage{subcaption}
\usepackage{comment}
\newcommand{\mat}[1]{\bm{#1}}
\newcommand{\ten}[1]{\bm{\mathcal{#1}}}

\newtheorem{example}{Example}
\newtheorem{algorithm}{Algorithm}[section]

\title{Computing low-rank approximations of large-scale matrices with the Tensor Network randomized SVD}

\author{Kim Batselier\thanks{Department of Electrical and Electronic Engineering, The University of Hong Kong, Hong Kong (kimb@eee.hku.hk)}
        \and Wenjian Yu\thanks{TNList, Department of Computer Science and Technology, Tsinghua University, Beijing 100084, China (yu-wj@tsinghua.edu.cn)} \and Luca Daniel\thanks{Department of Electrical Engineering and Computer Science, Massachusetts Institute of Technology (MIT), Cambridge, MA. (luca@mit.edu)} \and Ngai Wong \thanks{Department of Electrical and Electronic Engineering, The University of Hong Kong, Hong Kong (nwong@eee.hku.hk)}}

\begin{document}

\maketitle

\begin{abstract}
We propose a new algorithm for the computation of a singular value decomposition (SVD) low-rank approximation of a matrix in the Matrix Product Operator (MPO) format, also called the Tensor Train Matrix format. Our tensor network randomized SVD (TNrSVD) algorithm is an MPO implementation of the randomized SVD algorithm that is able to compute dominant singular values and their corresponding singular vectors. In contrast to the state-of-the-art tensor-based alternating least squares SVD (ALS-SVD) and modified alternating least squares SVD (MALS-SVD) matrix approximation methods, TNrSVD can be up to 17 times faster while achieving the same accuracy. In addition, our TNrSVD algorithm also produces accurate approximations in particular cases where both ALS-SVD and MALS-SVD fail to converge. We also propose a new algorithm for the fast conversion of a sparse matrix into its corresponding MPO form, which is up to 509 times faster than the standard Tensor Train SVD (TT-SVD) method while achieving machine precision accuracy. The efficiency and accuracy of both algorithms are demonstrated in numerical experiments.
\end{abstract}



\pagestyle{myheadings}
\thispagestyle{plain}
\markboth{KIM BATSELIER \textit{et al.}}{TNrSVD}

\section{Introduction}
\label{sec:intro}
When Beltrami established the existence of the singular value decomposition (SVD) in 1873~\cite{beltrami,moonen1995svd}, he probably had not foreseen that this matrix factorization would become a crucial tool in scientific computing and data analysis~\cite{bjorck1996,demmel1997,Elden2007,matrixcomputations}. Among the many applications of the SVD are the determination of the numerical rank and condition number of a matrix, the computation of low-rank approximations and pseudoinverses, and solving linear systems. These applications have found widespread usage in many fields of science and engineering~\cite{Elden2007,moonen1995svd}. Matrices with low numerical ranks appear in a wide variety of scientific applications~\cite{Halko2011}. For these matrices, finding a low-rank approximation allows them to be stored inexpensively without much loss of accuracy. It is not uncommon for matrices in data analysis to be very large and classical methods to compute the SVD~\cite{golub1965,Golub1970,matrixcomputations} can be ill-suited to handle large matrices. One proposed solution to compute a low-rank approximation of large data matrices is to use randomized algorithms~\cite{Halko2011}. An attractive feature of these algorithms is that they require only a constant number of passes over the data. It is even possible to find a matrix approximation with a single pass over the data\cite{Wenjian16}. This enables the efficient computation of a low-rank approximation of dense matrices that cannot be stored completely in fast memory~\cite{Wenjian17}.

Another way to handle large matrices is to use a different data storage representation. A Matrix Product Operator (MPO), also called Tensor Train Matrix~\cite{Oseledets2010}, is a particular tensor network representation of a matrix that originates from the simulation of one-dimensional quantum-lattice systems~\cite{schollwock2011density}. This representation transforms the storage complexity of an $n^d \times n^d$ matrix into $O(dn^2r^2)$, where $r$ is the maximal MPO-rank, effectively transforming the exponential dependence on $d$ into a linear one. The main idea of this representation is to replace the storage of a particular matrix element by a product of small matrices. An efficient representation is then found when only a few small matrices are required. A Matrix Product State (MPS), also called a Tensor Train~\cite{ivanTT}, is a similar tensor network representation of a vector. These tensor network representations have gained more interest over the past decade, together with their application to various problems~\cite{Batselier2016,Chen2016,MAL-059,MAL-067,Lee2016,Novikov2015,Os-mvk2-2011,Oseledets2012}. In particular, finding an SVD low-rank approximation of a matrix in the MPO representation is addressed in~\cite{Lee2015} with the development of the ALS-SVD and MALS-SVD methods. These two methods are shown to be able to compute a few extreme singular values of $2^{50} \times 2^{50}$ matrix accurately in a few seconds on desktop computers. In this article, we propose a randomized tensor network algorithm for the computation of an SVD low-rank approximation of a matrix. As we will demonstrate through numerical experiments, our proposed algorithm manages to achieve the same accuracy up to 17 times faster than MALS-SVD. Moreover, our algorithm is able to retrieve accurate approximations for cases where both ALS-SVD and MALS-SVD fail to converge. More specifically, the main contributions of this article are twofold:
\begin{enumerate}
\item We present a fast algorithm that is able to convert a given sparse matrix into MPO form with machine precision accuracy.
\item We present a MPO version of the randomized SVD algorithm that can outperform the current state-of-the-art tensor algorithms~\cite{Lee2015} for computing low-rank matrix approximations of large-scale matrices.\footnote{MATLAB implementations of all algorithms are distributed under a GNU lesser general public license and can be freely downloaded from~\url{https://github.com/kbatseli/TNrSVD}.} 
\end{enumerate}
This article is organized as follows. In Section \ref{sec:notation}, some basic tensor concepts and notation are explained. We introduce the notion of Matrix Product Operators in Section \ref{sec:mpt}. Our newly proposed algorithm to convert a matrix into the MPO representation is presented in Section \ref{sec:matrix2mpo}. In Section \ref{sec:TNrSVD}, we present our randomized algorithm to compute an SVD low-rank approximation of a given matrix in MPO form. Numerical experiments in Section \ref{sec:experiments} demonstrate both the fast matrix to MPO conversion as well as our randomized algorithm. We compare the performance of our matrix conversion algorithm with both the TT-SVD~\cite{ivanTT} and TT-cross~\cite{ttcross} algorithms, while also comparing the performance of our randomized algorithm with the ALS-SVD and MALS-SVD algorithms~\cite{Lee2015}. Section \ref{sec:conclusions} presents some conclusions together with an avenue of future research.

\section{Tensor basics and notation}
\label{sec:notation}
Tensors in this article are multi-dimensional arrays with entries either in the real or complex field. We denote scalars by italic letters, vectors by boldface italic letters, matrices by boldface capitalized italic letters and higher-order tensors by boldface calligraphic italic letters. The number of indices required to determine an entry of a tensor is called the order of the tensor. A $d$th order or $d$-way tensor is hence denoted $\ten{A} \in \mathbb{R}^{I_1 \times I_2 \times \cdots \times I_d}$. An index $i_k$ always satisfies $1\leq i_k \leq I_k$, where $I_k$ is called the dimension of that particular mode. We use the MATLAB array index notation to denote entries of tensors. Suppose that~$\ten{A}$ is a 4-way tensor with entries~$\ten{A}(i_1,i_2,i_3,i_4)$. Grouping indices together into multi-indices is one way of reshaping the tensor. For example, a $3$-way tensor can now be formed from $\ten{A}$ by grouping the first two indices together. The entries of this $3$-way tensor are then denoted by~$\ten{A}([i_1i_2],i_3,i_4)$, where the multi-index $[i_1i_2]$ is easily converted into a single index as $i_1+I_1(i_2-1)$. Grouping the indices into $[i_1]$ and $[i_2i_3i_4]$ results in a $I_1 \times I_2I_3I_4$ matrix with entries~$\ten{A}(i_1,[i_2i_3i_4])$. The column index $[i_2i_3i_4]$ is equivalent to the linear index $i_2+I_2 (i_3-1)+I_2 I_3 (i_4-1)$. More general, we define a multi-index $[i_1i_2\cdots i_d]$ as
\begin{align}
[i_1i_2\cdots i_d] := i_1 + \sum_{k=2}^{d}\,(i_k-1)\,\prod_{l=1}^{k-1} I_l.
\label{eq:multiindex}
\end{align}
Grouping indices together in order to change the order of a tensor is called reshaping and is an often used tensor operation. We adopt the MATLAB/Octave reshape operator ``reshape($\ten{A},[I_1,I_2, \cdots ,I_d])$", which reshapes the $d$-way tensor $\ten{A}$ into a tensor with dimensions $I_1 \times I_2 \times \cdots \times I_d$. The total number of elements of $\ten{A}$ must be the same as $I_1\times I_2 \times \cdots \times I_d$. The mode-$n$ matricization~$\ten{A}_{(n)}$ of a $d$-way tensor~$\ten{A}$ maps the entry $\ten{A}(i_1,i_2,\cdots,i_d)$ to the matrix element with row index $i_n$ and column index $[i_1\cdots i_{n-1}i_{n+1}\cdots i_d]$.
\begin{example}
\label{ex:ex1}
We illustrate the reshaping operator on the $4\times 3 \times 2$ tensor $\ten{A}$ that contains all entries from 1 up to 24. Its mode-1 matricization is
\begin{align*}
\ten{A}_{(1)} &= \textrm{reshape}(\ten{A},[4,6]) = 
\begin{pmatrix}
1 & 5 & 9 & 13 & 17 & 21\\
2 & 6 & 10 & 14 &18 & 22\\
3 & 7 & 11 &15 &19 & 23\\
4 & 8 & 12 & 16&20 & 24
\end{pmatrix}.
\end{align*}
\end{example}
Another important reshaping of a tensor $\ten{A}$ is its vectorization, denoted $\textrm{vec}(\ten{A})$ and obtained from grouping all indices into one multi-index.
\begin{example}
For the tensor of Example \ref{ex:ex1}, we have 
\begin{align*}
\textrm{vec}(\ten{A}) \; =\textrm{reshape}(\ten{A},[24,1])=\; \begin{pmatrix}
1& 2& \cdots &24
\end{pmatrix}^T.
\end{align*}
\end{example}

Suppose we have two $d$-way tensors $\ten{A} \in \mathbb{R}^{I_1 \times I_2 \times \cdots \times I_d},\mat{B} \in \mathbb{R}^{J_1 \times J_2 \times \cdots \times J_d}$. The Kronecker product $\ten{C}=\ten{A} \otimes \ten{B} \in \mathbb{R}^{I_1J_1 \times I_2J_2 \times \cdots \times I_dJ_d}$ is then a $d$-way tensor such that
\begin{align}
\ten{C}([j_1i_1],[j_2i_2],\ldots,[j_di_d]) &= \ten{A}(i_1,i_2,\ldots,i_d)\, \ten{B}(j_1,j_2,\ldots,j_d).
\label{eq:kron}
\end{align}
Similarly, the outer product $\ten{D}=\ten{A} \circ \ten{B}$ of the $d$-way tensors $\ten{A},\ten{B}$ is a $2d$-way tensor of dimensions $I_1 \times \cdots \times I_d \times J_1 \times \cdots \times J_d$ such that
\begin{align}
\ten{D}(i_1,i_2,\ldots,i_d,j_1,j_2,\ldots,j_d)&= \ten{A}(i_1,i_2,\ldots,i_d)\, \ten{B}(j_1,j_2,\ldots,j_d).
\label{eq:outerprod}
\end{align}
From equations \eqref{eq:kron} and \eqref{eq:outerprod} one can see that the Kronecker and outer products are interrelated through a reshaping and a permutation of the indices. A very convenient graphical representation of $d$-way tensors is shown in Figure \ref{fig:TNgraphs}. Tensors are here represented by circles and each `leg' denotes a particular mode of the tensor. The order of the tensor is then easily determined by counting the number of legs. Since a scalar is a zeroth-order tensor, it is represented by a circle without any lines.  
\begin{figure}[th]
\begin{center}
\ifx\du\undefined
  \newlength{\du}
\fi
\setlength{\du}{8\unitlength}
\begin{tikzpicture}
\pgftransformxscale{1.000000}
\pgftransformyscale{-1.000000}
\definecolor{dialinecolor}{rgb}{0.000000, 0.000000, 0.000000}
\pgfsetstrokecolor{dialinecolor}
\definecolor{dialinecolor}{rgb}{1.000000, 1.000000, 1.000000}
\pgfsetfillcolor{dialinecolor}
\definecolor{dialinecolor}{rgb}{1.000000, 1.000000, 1.000000}
\pgfsetfillcolor{dialinecolor}
\pgfpathellipse{\pgfpoint{-6.073446\du}{10.779091\du}}{\pgfpoint{2.900000\du}{0\du}}{\pgfpoint{0\du}{2.800000\du}}
\pgfusepath{fill}
\pgfsetlinewidth{0.100000\du}
\pgfsetdash{}{0pt}
\pgfsetdash{}{0pt}
\definecolor{dialinecolor}{rgb}{0.000000, 0.000000, 0.000000}
\pgfsetstrokecolor{dialinecolor}
\pgfpathellipse{\pgfpoint{-6.073446\du}{10.779091\du}}{\pgfpoint{2.900000\du}{0\du}}{\pgfpoint{0\du}{2.800000\du}}
\pgfusepath{stroke}
\pgfsetlinewidth{0.100000\du}
\pgfsetdash{}{0pt}
\pgfsetdash{}{0pt}
\pgfsetbuttcap
{
\definecolor{dialinecolor}{rgb}{0.000000, 0.000000, 0.000000}
\pgfsetfillcolor{dialinecolor}
\definecolor{dialinecolor}{rgb}{0.000000, 0.000000, 0.000000}
\pgfsetstrokecolor{dialinecolor}
\draw (-4.022836\du,12.758990\du)--(-4.022927\du,17.929867\du);
}
\pgfsetlinewidth{0.100000\du}
\pgfsetdash{}{0pt}
\pgfsetdash{}{0pt}
\pgfsetbuttcap
{
\definecolor{dialinecolor}{rgb}{0.000000, 0.000000, 0.000000}
\pgfsetfillcolor{dialinecolor}
\definecolor{dialinecolor}{rgb}{0.000000, 0.000000, 0.000000}
\pgfsetstrokecolor{dialinecolor}
\draw (-8.124055\du,12.758990\du)--(-8.110342\du,17.972712\du);
}
\pgfsetlinewidth{0.100000\du}
\pgfsetdash{}{0pt}
\pgfsetdash{}{0pt}
\pgfsetbuttcap
{
\definecolor{dialinecolor}{rgb}{0.000000, 0.000000, 0.000000}
\pgfsetfillcolor{dialinecolor}
\definecolor{dialinecolor}{rgb}{0.000000, 0.000000, 0.000000}
\pgfsetstrokecolor{dialinecolor}
\draw (-6.073446\du,13.579091\du)--(-6.070919\du,17.947005\du);
}
\definecolor{dialinecolor}{rgb}{0.000000, 0.000000, 0.000000}
\pgfsetstrokecolor{dialinecolor}
\node[anchor=west] at (-7.2\du,10.6\du){$\ten{A}$};
\definecolor{dialinecolor}{rgb}{1.000000, 1.000000, 1.000000}
\pgfsetfillcolor{dialinecolor}
\pgfpathellipse{\pgfpoint{-36.297515\du}{10.779091\du}}{\pgfpoint{2.900000\du}{0\du}}{\pgfpoint{0\du}{2.800000\du}}
\pgfusepath{fill}
\pgfsetlinewidth{0.100000\du}
\pgfsetdash{}{0pt}
\pgfsetdash{}{0pt}
\definecolor{dialinecolor}{rgb}{0.000000, 0.000000, 0.000000}
\pgfsetstrokecolor{dialinecolor}
\pgfpathellipse{\pgfpoint{-36.297515\du}{10.779091\du}}{\pgfpoint{2.900000\du}{0\du}}{\pgfpoint{0\du}{2.800000\du}}
\pgfusepath{stroke}
\definecolor{dialinecolor}{rgb}{0.000000, 0.000000, 0.000000}
\pgfsetstrokecolor{dialinecolor}
\node[anchor=west] at (-37.235025\du,10.6\du){$a$};
\definecolor{dialinecolor}{rgb}{1.000000, 1.000000, 1.000000}
\pgfsetfillcolor{dialinecolor}
\pgfpathellipse{\pgfpoint{-26.560105\du}{10.779091\du}}{\pgfpoint{2.900000\du}{0\du}}{\pgfpoint{0\du}{2.800000\du}}
\pgfusepath{fill}
\pgfsetlinewidth{0.100000\du}
\pgfsetdash{}{0pt}
\pgfsetdash{}{0pt}
\definecolor{dialinecolor}{rgb}{0.000000, 0.000000, 0.000000}
\pgfsetstrokecolor{dialinecolor}
\pgfpathellipse{\pgfpoint{-26.560105\du}{10.779091\du}}{\pgfpoint{2.900000\du}{0\du}}{\pgfpoint{0\du}{2.800000\du}}
\pgfusepath{stroke}
\pgfsetlinewidth{0.100000\du}
\pgfsetdash{}{0pt}
\pgfsetdash{}{0pt}
\pgfsetbuttcap
{
\definecolor{dialinecolor}{rgb}{0.000000, 0.000000, 0.000000}
\pgfsetfillcolor{dialinecolor}
\definecolor{dialinecolor}{rgb}{0.000000, 0.000000, 0.000000}
\pgfsetstrokecolor{dialinecolor}
\draw (-26.560105\du,13.579091\du)--(-26.557578\du,17.824077\du);
}
\definecolor{dialinecolor}{rgb}{0.000000, 0.000000, 0.000000}
\pgfsetstrokecolor{dialinecolor}
\node[anchor=west] at (-27.497615\du,10.6\du){$\mat{a}$};
\definecolor{dialinecolor}{rgb}{1.000000, 1.000000, 1.000000}
\pgfsetfillcolor{dialinecolor}
\pgfpathellipse{\pgfpoint{-16.948419\du}{10.779091\du}}{\pgfpoint{2.900000\du}{0\du}}{\pgfpoint{0\du}{2.800000\du}}
\pgfusepath{fill}
\pgfsetlinewidth{0.100000\du}
\pgfsetdash{}{0pt}
\pgfsetdash{}{0pt}
\definecolor{dialinecolor}{rgb}{0.000000, 0.000000, 0.000000}
\pgfsetstrokecolor{dialinecolor}
\pgfpathellipse{\pgfpoint{-16.948419\du}{10.779091\du}}{\pgfpoint{2.900000\du}{0\du}}{\pgfpoint{0\du}{2.800000\du}}
\pgfusepath{stroke}
\pgfsetlinewidth{0.100000\du}
\pgfsetdash{}{0pt}
\pgfsetdash{}{0pt}
\pgfsetbuttcap
{
\definecolor{dialinecolor}{rgb}{0.000000, 0.000000, 0.000000}
\pgfsetfillcolor{dialinecolor}
\definecolor{dialinecolor}{rgb}{0.000000, 0.000000, 0.000000}
\pgfsetstrokecolor{dialinecolor}
\draw (-14.897810\du,12.758990\du)--(-14.897900\du,18.070961\du);
}
\pgfsetlinewidth{0.100000\du}
\pgfsetdash{}{0pt}
\pgfsetdash{}{0pt}
\pgfsetbuttcap
{
\definecolor{dialinecolor}{rgb}{0.000000, 0.000000, 0.000000}
\pgfsetfillcolor{dialinecolor}
\definecolor{dialinecolor}{rgb}{0.000000, 0.000000, 0.000000}
\pgfsetstrokecolor{dialinecolor}
\draw (-18.999029\du,12.758990\du)--(-18.985315\du,18.113806\du);
}
\definecolor{dialinecolor}{rgb}{0.000000, 0.000000, 0.000000}
\pgfsetstrokecolor{dialinecolor}
\node[anchor=west] at (-17.885929\du,10.6\du){$\mat{A}$};
\end{tikzpicture}
\end{center}
\caption{Graphical depiction of a scalar $a$, vector $\mat{a}$, matrix $\mat{A}$ and 3-way tensor $\ten{A}$.}
\label{fig:TNgraphs}
\end{figure}
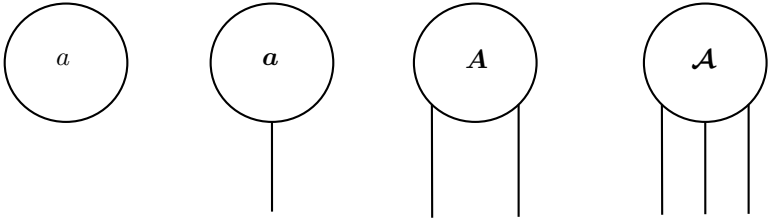
One of the most important operations on tensors is the summation over indices, also called contraction of indices. For example, the following mode product~\cite{tensorreview} of a 3-way tensor $\ten{A} \in \mathbb{R}^{I \times I \times I}$ with a matrix $\mat{U}_1 \in \mathbb{R}^{R \times I}$ and a vector $\mat{u}_3 \in \mathbb{R}^I$
\begin{align*}
\ten{A} \, \times_1\, \mat{U}_1 \, \times_3 \,\mat{u}_3^T = \sum_{i,j} \, \ten{A}(i,:,j) \; \mat{U}_1(:,i) \; \mat{u}_3(j)
\end{align*}
is graphically depicted in Figure \ref{fig:TNcontraction} by connected lines between $\ten{A},\mat{U}_1$ and $\mat{u}_3$.
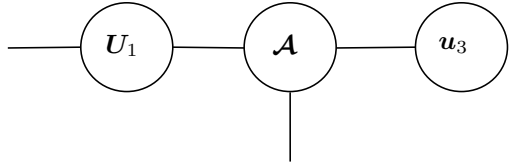
\begin{figure}[th]
\begin{center}
\ifx\du\undefined
  \newlength{\du}
\fi
\setlength{\du}{6\unitlength}
\begin{tikzpicture}
\pgftransformxscale{1.000000}
\pgftransformyscale{-1.000000}
\definecolor{dialinecolor}{rgb}{0.000000, 0.000000, 0.000000}
\pgfsetstrokecolor{dialinecolor}
\definecolor{dialinecolor}{rgb}{1.000000, 1.000000, 1.000000}
\pgfsetfillcolor{dialinecolor}
\pgfsetlinewidth{0.100000\du}
\pgfsetdash{}{0pt}
\pgfsetdash{}{0pt}
\pgfsetbuttcap
{
\definecolor{dialinecolor}{rgb}{0.000000, 0.000000, 0.000000}
\pgfsetfillcolor{dialinecolor}
\definecolor{dialinecolor}{rgb}{0.000000, 0.000000, 0.000000}
\pgfsetstrokecolor{dialinecolor}
\draw (17.845000\du,8.660000\du)--(13.345000\du,8.610000\du);
}
\pgfsetlinewidth{0.100000\du}
\pgfsetdash{}{0pt}
\pgfsetdash{}{0pt}
\pgfsetbuttcap
{
\definecolor{dialinecolor}{rgb}{0.000000, 0.000000, 0.000000}
\pgfsetfillcolor{dialinecolor}
\definecolor{dialinecolor}{rgb}{0.000000, 0.000000, 0.000000}
\pgfsetstrokecolor{dialinecolor}
\draw (20.745000\du,11.460000\du)--(20.747527\du,15.827914\du);
}
\definecolor{dialinecolor}{rgb}{1.000000, 1.000000, 1.000000}
\pgfsetfillcolor{dialinecolor}
\pgfpathellipse{\pgfpoint{20.745000\du}{8.610000\du}}{\pgfpoint{2.900000\du}{0\du}}{\pgfpoint{0\du}{2.800000\du}}
\pgfusepath{fill}
\pgfsetlinewidth{0.100000\du}
\pgfsetdash{}{0pt}
\pgfsetdash{}{0pt}
\definecolor{dialinecolor}{rgb}{0.000000, 0.000000, 0.000000}
\pgfsetstrokecolor{dialinecolor}
\pgfpathellipse{\pgfpoint{20.745000\du}{8.610000\du}}{\pgfpoint{2.900000\du}{0\du}}{\pgfpoint{0\du}{2.800000\du}}
\pgfusepath{stroke}
\definecolor{dialinecolor}{rgb}{0.000000, 0.000000, 0.000000}
\pgfsetstrokecolor{dialinecolor}
\node[anchor=west] at (19.\du,8.5\du){$\ten{A}$};
\definecolor{dialinecolor}{rgb}{1.000000, 1.000000, 1.000000}
\pgfsetfillcolor{dialinecolor}
\pgfpathellipse{\pgfpoint{10.445000\du}{8.610000\du}}{\pgfpoint{2.900000\du}{0\du}}{\pgfpoint{0\du}{2.800000\du}}
\pgfusepath{fill}
\pgfsetlinewidth{0.100000\du}
\pgfsetdash{}{0pt}
\pgfsetdash{}{0pt}
\definecolor{dialinecolor}{rgb}{0.000000, 0.000000, 0.000000}
\pgfsetstrokecolor{dialinecolor}
\pgfpathellipse{\pgfpoint{10.445000\du}{8.610000\du}}{\pgfpoint{2.900000\du}{0\du}}{\pgfpoint{0\du}{2.800000\du}}
\pgfusepath{stroke}
\definecolor{dialinecolor}{rgb}{0.000000, 0.000000, 0.000000}
\pgfsetstrokecolor{dialinecolor}
\node[anchor=west] at (8.5\du,8.5\du){$\mat{U}_1$};
\pgfsetlinewidth{0.100000\du}
\pgfsetdash{}{0pt}
\pgfsetdash{}{0pt}
\pgfsetbuttcap
{
\definecolor{dialinecolor}{rgb}{0.000000, 0.000000, 0.000000}
\pgfsetfillcolor{dialinecolor}
\definecolor{dialinecolor}{rgb}{0.000000, 0.000000, 0.000000}
\pgfsetstrokecolor{dialinecolor}
\draw (7.545000\du,8.610000\du)--(2.945552\du,8.660552\du);
}
\pgfsetlinewidth{0.100000\du}
\pgfsetdash{}{0pt}
\pgfsetdash{}{0pt}
\pgfsetbuttcap
{
\definecolor{dialinecolor}{rgb}{0.000000, 0.000000, 0.000000}
\pgfsetfillcolor{dialinecolor}
\definecolor{dialinecolor}{rgb}{0.000000, 0.000000, 0.000000}
\pgfsetstrokecolor{dialinecolor}
\draw (28.645000\du,8.610000\du)--(23.645000\du,8.660000\du);
}
\definecolor{dialinecolor}{rgb}{1.000000, 1.000000, 1.000000}
\pgfsetfillcolor{dialinecolor}
\pgfpathellipse{\pgfpoint{31.545000\du}{8.610000\du}}{\pgfpoint{2.900000\du}{0\du}}{\pgfpoint{0\du}{2.800000\du}}
\pgfusepath{fill}
\pgfsetlinewidth{0.100000\du}
\pgfsetdash{}{0pt}
\pgfsetdash{}{0pt}
\definecolor{dialinecolor}{rgb}{0.000000, 0.000000, 0.000000}
\pgfsetstrokecolor{dialinecolor}
\pgfpathellipse{\pgfpoint{31.545000\du}{8.610000\du}}{\pgfpoint{2.900000\du}{0\du}}{\pgfpoint{0\du}{2.800000\du}}
\pgfusepath{stroke}
\definecolor{dialinecolor}{rgb}{0.000000, 0.000000, 0.000000}
\pgfsetstrokecolor{dialinecolor}
\node[anchor=west] at (29.5\du,8.5\du){$\mat{u}_3$};
\end{tikzpicture}
\end{center}
\caption{Summation over the first and third index of $\ten{A}$ represented by connected lines in the tensor network graph.}
\label{fig:TNcontraction}
\end{figure}
Figure \ref{fig:TNcontraction} also illustrates a simple tensor network, which is a collection of tensors that are interconnected through contractions. The tensor network in Figure \ref{fig:TNcontraction} has two legs, which indicates that the network represents a matrix. This article uses a very particular tensor network structure, the Matrix Product Operator structure.

\section{Matrix Product Operators}
\label{sec:mpt}
In this section, we give a brief introduction to the notion of MPOs. Simply put, an MPO is a linear chain of 4-way tensors that represents a matrix and was originally used to represent an operator acting on a multi-body quantum system. Since their introduction to the scientific community in 2010~\cite{Oseledets2010}, MPOs have found many other applications. We now discuss the MPO representation of a matrix through an illustrative example. Suppose that we have a matrix $\mat{A}$ of size $I_1I_2I_3I_4 \times J_1J_2J_3J_4$, as shown in Figure \ref{fig:MPO}.
\begin{figure}[h]
\begin{center}
\input{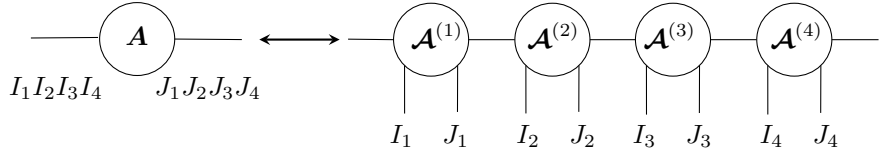}
\end{center}
\caption{Representation of an $I_1I_2I_3I_4\times J_1J_2J_3J_4$ matrix $\mat{A}$ as an MPO.}
\label{fig:MPO}
\end{figure}
This matrix $\mat{A}$ can be represented by an MPO of four 4-way tensors, where the first tensor has dimensions $R_1 \times I_1 \times J_1 \times R_2$. Similarly, the $k$th tensor in the MPO of Figure~\ref{fig:MPO} hence has dimensions $R_k \times I_k \times J_k \times R_{k+1}$. We require that $R_1=R_5=1$, which ensures that the contraction of this particular MPO results in a 8-way tensor with entries $\ten{A}(i_1,j_1,i_2,j_2,i_3,j_3,i_4,j_4)$. This tensor can then be permuted and reshaped back into the original matrix with entries $\mat{A}([i_1i_2i_3i_4],[j_1j_2j_3j_4])$.
The dimensions $R_k,R_{k+1}$ of the connecting indices in an MPO are called the MPO-ranks and play a crucial role in the computational complexity of our developed algorithms. The MPO-ranks are called canonical if they attain their minimal value such that the MPO represents a given matrix exactly. A very special MPO is obtained when all MPO-ranks are unity. The contraction of a rank-one MPO corresponds with the outer product of the individual MPO-tensors. Indeed, suppose that the MPO-ranks in Figure \ref{fig:MPO} are all unity. The 4-way tensors of the MPO are then reduced to matrices such that we can write $\mat{A}^{(k)}:=\ten{A}^{(k)}(1,:,:,1)\,(k=1,\ldots,4)$ and
\begin{align*}
\ten{A}(i_1,j_1,i_2,j_2,i_3,j_3,i_4j_4) = \mat{A}^{(1)}(i_1,j_1)\,\mat{A}^{(2)}(i_2,j_2)\,\mat{A}^{(3)}(i_3,j_3)\,\mat{A}^{(4)}(i_4,j_4),
\end{align*}
which is exactly the outer product of the matrices $\mat{A}^{(1)},\mat{A}^{(2)},\mat{A}^{(3)}$ with $\mat{A}^{(4)}$ into a 8-way tensor. The relation between the Kronecker and outer products, together with the previous example leads to the following important theorem.
\begin{theorem}
\label{theo:unitMPT}
A matrix $\mat{A} \in \mathbb{R}^{I_1I_2\cdots I_d \times J_1J_2\cdots J_d }$ that satisfies
\begin{align*}
\mat{A} &= \mat{A}^{(d)} \otimes \cdots \otimes \mat{A}^{(2)} \otimes \mat{A}^{(1)} 
\end{align*}
has an MPO representation where the $k$th MPO-tensor is $\mat{A}^{(k)} \in \mathbb{R}^{1 \times I_k \times J_k \times 1}\, (k=1,\ldots,d)$ with unit canonical MPO-ranks.
\end{theorem}

It is important to note that the order of the MPO-tensors is reversed with respect to the order of the factor matrices in the Kronecker product. This means that the last factor matrix $\mat{A}^{(1)} $ in the Kronecker product of Theorem \ref{theo:unitMPT} is the first tensor in the corresponding MPO representation. Theorem \ref{theo:unitMPT} can also be written in terms of the matrix entries as
\begin{align*}
\mat{A}([i_1i_2\cdots i_d],[j_1j_2\cdots j_d]) &= \mat{A}^{(d)}(i_d,j_d) \cdots \mat{A}^{(2)}(i_2,j_2)\,\mat{A}^{(1)}(i_1,j_1).
\end{align*}
As mentioned earlier, the MPO-ranks play a crucial role in the computational complexity of the algorithms. For this reason, only MPOs with small ranks are desired. An upper bound on the canonical MPO-rank $R_k$ for an MPO of $d$ tensors for which $R_1=R_{d+1}=1$ is given by the following theorem.
\begin{theorem}(Modified version of Theorem 2.1 in \cite{ttcross})
\label{theo:MPTranks}
For any matrix $\mat{A} \in \mathbb{R}^{I_1I_2\cdots I_d \times J_1J_2\cdots J_d }$ there exists an MPO with MPO-ranks $R_1=R_{d+1}=1$ such that the canonical MPO-ranks $R_k$ satisfy
\begin{align*}
R_k \leq \textrm{min}\,\left( \prod_{i=1}^{k-1} I_iJ_i, \prod_{i=k}^d I_iJ_i \right) \textrm{ for } k=2,\ldots,d.
\end{align*}
\end{theorem}
\begin{proof}
The upper bound on the canonical MPO-rank $R_k$ can be determined from contracting the first $k-1$ tensors of the MPO together and reshape the result into the $I_1J_1\cdots I_{k-1}J_{k-1} \times R_k$ matrix $\mat{B}$. Similarly, we can contract the $k$th tensor of the MPO with all other remaining tensors and reshape the result into the $R_k \times I_{k}J_k\cdots I_dJ_d$ matrix $\mat{C}$. The canonical rank $R_k$ is now upper bounded by the matrix product $\mat{B}\mat{C}$ due to $\textrm{rank}(\mat{B}\mat{C}) \leq \textrm{min}(\textrm{rank}(\mat{B}),\textrm{rank}(\mat{C}))$. 
\end{proof}

These upper bounds are quite pessimistic and are attained for generic tensors. For example, a generic full-rank $I^{10} \times I^{10}$ matrix has an exact MPO representation with a canonical MPO-rank $R_6=I^{10}$. This implies that any MPO representation with $R_6 < I^{10}$ will consequently be an approximation of the original matrix. Theorem \ref{theo:MPTranks} therefore allows us to conclude that MPOs are only useful when either the canonical MPO-ranks are small or a low-rank MPO exists that approximates the underlying matrix sufficiently.

Two important MPO operations are addition and rounding, which are easily generalized from MPS addition and rounding. Indeed, by grouping the $I_k$ and $J_k$ indices together, one effectively transforms the MPO into an MPS such that MPS addition and rounding can be applied. The addition of two MPOs results in an MPO for which the corresponding ranks, except $R_1$ and $R_{d+1}$, are added. The rounding operation repeatedly uses the singular value decomposition (SVD) on each of the MPO-tensors going from left-to-right and right-to-left in order to truncate the ranks $R_k$ such that a specific relative error tolerance is satisfied. For more details on MPS addition and rounding we would like to refer the reader to~\cite[p.~2305]{ivanTT} and \cite[p.~2308]{ivanTT}, respectively. An alternative MPO rounding operation that does not require the SVD but relies on removing parallel vectors is described in~\cite{Hubig2017}. This alternative rounding procedure can be computationally more efficient for MPOs that consist of sparse tensors.

\section{Converting a sparse matrix into an MPO}
\label{sec:matrix2mpo}
The standard way to convert a matrix into MPO form is the TT-SVD algorithm~\cite[p.~2301]{ivanTT}, which relies on consecutive reshapings of the matrix from which an SVD needs to be computed. This procedure is not recommended for matrices in real applications for two reasons. First, application-specific matrices tend to be sparse and computing the SVD of a sparse matrix destroys the sparsity, which results in requiring more and often prohibitive storage. Second, real life matrices are typically so large that it is infeasible to compute their SVD. An alternative method to convert a matrix into an MPO is via cross approximation~\cite{ttcross}. This method relies on heuristics to find subsets of indices for all modes of a tensor in order to approximate it. In practice, the cross approximation method can be very slow and not attain the desired accuracy. Our matrix to MPO conversion method relies on a partitioning of the sparse matrix such that it is easily written as the addition of rank-1 MPOs.
\subsection{Algorithm derivation}
We derive our algorithm with the following illustrative example. Suppose we have a sparse matrix $\mat{A} \in \mathbb{R}^{I \times J}$ with the following sparsity pattern
\begin{align*}
\begin{pmatrix}
0 & 0 & 0 & \mat{A}_{14} \\
0 & \mat{A}_{22} & 0 & 0 \\
0 & 0 & \mat{A}_{33} & 0 \\
0 & 0 & 0 & 0
\end{pmatrix}.
\end{align*}
Assume that each of the nonzero block matrices has dimensions $I_1 \times J_1$ and that $I=I_1\,I_2,J=J_1\,J_2$ such that the rows and columns of $\mat{A}$ are now indexed by $[i_1i_2],[j_1j_2]$, respectively. The main idea of our method is to convert each nonzero block matrix into a rank-1 MPO and add them all together. Observe now that
\begin{align*}
\begin{pmatrix}
0 & 0 & 0 & \mat{A}_{14} \\
0 & 0 & 0 & 0 \\
0 & 0 & 0 & 0 \\
0 & 0 & 0 & 0
\end{pmatrix} &= \mat{E}_{14} \otimes \mat{A}_{14},
\end{align*}
where $\mat{E}_{14} \in \mathbb{R}^{I_2 \times J_2}$ is a matrix of zeros except for $\mat{E}_{14}(1,4)=1$. From Theorem \ref{theo:unitMPT} we know that $\mat{E}_{14} \otimes \mat{A}_{14}$ is equivalent with a rank-1 MPO where the first MPO-tensor is $\mat{A}_{14}$ and the second MPO-tensor is $\mat{E}_{14}$. Generalizing the matrix $\mat{E}_{14}$ to the matrix $\mat{E}_{ij}$ of zero entries except for $\mat{E}_{ij}(i,j)=1$ allows us to write
\begin{align}
 \label{ex:sumofkron2}
\mat{A} &= \mat{E}_{14} \otimes \mat{A}_{14} + \mat{E}_{22} \otimes \mat{A}_{22} + \mat{E}_{33} \otimes \mat{A}_{33},
\end{align}
from which we conclude that the MPO representation of $\mat{A}$ is found from adding the unit-rank MPOs of each of the terms. Another important conclusion is that the MPO-rank for the particular MPO obtained from this algorithm is the total number of summations. The number of factors in the Kronecker product is not limited to two and depends on the matrix partitioning. Indeed, suppose we can partition $\mat{A}_{14}$ further into 
\begin{align*}
\mat{A}_{14} &= 
\begin{pmatrix}
0 & 0 \\
\mat{X}_{14}  & 0
\end{pmatrix}  = \mat{E}_{21} \otimes \mat{X}_{14},
\end{align*}
then the first term of \eqref{ex:sumofkron2} becomes $\mat{E}_{14} \otimes \mat{E}_{21} \otimes \mat{X}_{14}$ and likewise for the other terms. A crucial element is that the matrix $\mat{A}$ is partitioned into block matrices of equal size, which is required for the addition of the MPOs. In general, for a given matrix $\mat{A} \in \mathbb{R}^{I \times J}$, we consider the partitioning of $\mat{A}$ determined by a Kronecker product of $d$ matrices
\begin{align*}
\mat{A}^{(d)} \otimes \mat{A}^{(d-1)} \otimes \cdots  \otimes \mat{A}^{(2)} \otimes \mat{A}^{(1)}
\end{align*}
with $\mat{A}^{(k)} \in \mathbb{R}^{I_k \times J_k}\,(k=1,\ldots,d)$ and $I=\prod_{k=1}^d\,I_k,J=\prod_{k=1}^d\,J_k$. The algorithm to convert a sparse matrix into an MPO is given in pseudo-code in Algorithm \ref{alg:matrix2MPO}.\\
\\
\framebox[.99\textwidth][l]
{\begin{minipage}{0.99\textwidth}
\begin{algorithm}Sparse matrix to MPO conversion\\
\textit{\textbf{Input}}: matrix $\mat{A}$, dimensions $I_1,\ldots,I_d,J_1,\ldots,J_d$.\\
\textit{\textbf{Output}}:\makebox[0pt][l]{ MPO $\ten{A}$ with tensors $\ten{A}^{(1)},\ldots,\ten{A}^{(d)}$.}
\begin{algorithmic}
\STATE Initialize MPO $\ten{A}$ with zero tensors.
\FOR {all nonzero matrix blocks $\mat{X}$}
\STATE Determine $d-1$ $\mat{E}_{ij}$ matrices.
\STATE Construct rank-1 MPO $\ten{T}$ with $\mat{X}$ and $\mat{E}_{ij}$ matrices.
\STATE $\ten{A} \gets \ten{A} + \ten{T}$
\ENDFOR
\end{algorithmic}
\label{alg:matrix2MPO}
\end{algorithm}
\end{minipage}}\\
\subsection{Algorithm properties}
Having derived our sparse matrix conversion algorithm, we now discuss some of its properties. Algorithm \ref{alg:matrix2MPO} has the following nice features, some of which we will address in more detail:
\begin{itemize}
\item Except for $\ten{A}^{(1)}$, almost all of the MPO-tensors will be sparse.
\item The user is completely free to decide on how to partition the matrix $\mat{A}$, which determines the number of tensors in the resulting MPO.
\item The generalization of Algorithm \ref{alg:matrix2MPO} to construct an MPO representation of a given tensor is straightforward.
\item The maximal number of tensors in an MPO representation are easily deduced and given in Lemma \ref{lemma:maxd}.
\item The obtained MPO-rank for a particular partitioning of the matrix $\mat{A}$ is also easily deduced and given in Lemma \ref{lemma:MPOrank}.
\item A lower bound on the obtained MPO-rank for a fixed block size $I_1,J_1$ is derived in Lemma \ref{lemma:lowerR}.
\item As the dimensions of each of the MPO-tensors are known a priori, one can preallocate the required memory to store the tensors in advance. This allows a fast execution of Algorithm \ref{alg:matrix2MPO}.
\end{itemize}
The maximal number of tensors in an MPO representation of a matrix is determined by choosing a partitioning such that each block matrix of $\mat{A}$ becomes a single scalar entry and is given by the following lemma.

\begin{lemma}
\label{lemma:maxd}
Given a matrix $\mat{A} \in \mathbb{R}^{I \times J}$, suppose $d_I,d_J$ are the number of factors in the prime factorizations of $I,J$, respectively. Then the maximal number of tensors in an MPO representation of $\mat{A}$ is $\textrm{max}\,(d_I,d_J)+1$.
\end{lemma}

The following simple example illustrates the maximal number of tensors $\textrm{max}(d_I,d_J)+1$ from Lemma \ref{lemma:maxd}.
\begin{example}
Let
\begin{align*}
\mat{A} &= \begin{pmatrix}
2 & 0 & 0 & 0 & 0 & 0\\
0 & 0 & 0 & -5 & 0 & 0
\end{pmatrix} \in \mathbb{R}^{2 \times 6}.
\end{align*}
Then the prime factorizations are $2=2$ and $6=2\times 3$, which sets $d_I=1,d_J=2$ and the maximal number of tensors in the MPO of $\mat{A}$ is $\textrm{max}(1,2)+1=2+1=3$. Indeed, by setting $I_1=1,J_1=1,I_2=1,J_2=3,I_3=2,J_3=2$ we can write
\begin{align*}
\mat{A} &=
\begin{pmatrix}
1 & 0 \\
0 & 0
\end{pmatrix} \otimes \begin{pmatrix}
1 & 0 & 0 \end{pmatrix} \otimes 2 + \begin{pmatrix}
0 & 0 \\
0 & 1
\end{pmatrix} \otimes \begin{pmatrix}
1 & 0 & 0 \end{pmatrix} \otimes -5.
\end{align*}
We therefore have $\ten{A}^{(1)} \in \mathbb{R}^{1 \times 1 \times 2}, \ten{A}^{(2)} \in \mathbb{R}^{2 \times 1 \times 3 \times 2}, \ten{A}^{(3)} \in \mathbb{R}^{2 \times 2 \times 2 \times 1}$. Theorem \ref{theo:MPTranks} states that the canonical MPO-ranks satisfy $R_2 \leq 1, R_3\leq 3$, which demonstrates that the MPO-ranks obtained from Algorithm \ref{alg:matrix2MPO} are not necessarily minimal. 
\end{example}

The rank of the MPO obtained from Algorithm \ref{alg:matrix2MPO} for a particular partitioning of the matrix $\mat{A}$ is given by the following lemma.
\begin{lemma}
\label{lemma:MPOrank}
The MPO obtained from Algorithm \ref{alg:matrix2MPO} has a uniform MPO-rank equal to the total number of nonzero matrix blocks $\mat{X}$ as determined by the partitioning of $\mat{A}$. 
\end{lemma}

Lemma \ref{lemma:MPOrank} follows trivially from the fact that ranks are added in MPO addition and all MPOs in Algorithm \ref{alg:matrix2MPO} are unit-rank. It is important to realize that the usage of Algorithm \ref{alg:matrix2MPO} is not limited to sparse matrices \textit{per se}. One could apply Algorithm \ref{alg:matrix2MPO} to dense matrices but then the possible computational benefit of having to process only a few nonzero matrix blocks $\mat{X}$ is lost.  It is also the case that the MPO-ranks can be reduced in almost all cases via a rounding procedure without the loss of any accuracy, as the upper bounds of Theorem \ref{theo:MPTranks} are usually exceeded. Partitioning the matrix $\mat{A}$ such that each term in Algorithm \ref{alg:matrix2MPO} corresponds with a single scalar entry sets the resulting MPO-rank to the total number of nonzero entries of $\mat{A}$. This might be too high in practice. On the other hand, choosing any $I_k,J_k$ too large results in a large MPO-tensor $\ten{A}^{(k)}$, which is also not desired. A strategy that can work particularly well is to use the Cuthill-Mckee algorithm~\cite{Cuthill1969} to permute $\mat{A}$ into a banded matrix with a small bandwidth. Grouping all nonzero entries together around the main diagonal also effectively reduces the number of nonzero block matrices and hence the total MPO-rank. A block size $I_1 \times J_1$ can then be chosen such that the bandwidth is covered by a few blocks. Algorithm \ref{alg:matrix2MPO} can then be applied to the permuted matrix. Other permutations may reduce the maximal MPO-rank even further. We will discuss choosing the partitioning of $\mat{A}$ in more detail in Section \ref{subsec:partition}.

Algorithm \ref{alg:matrix2MPO} will construct an MPO with a uniform MPO-rank, which will exceed the upper bounds from Theorem \ref{theo:MPTranks} in almost all cases. One can use a rounding step to truncate the MPO-ranks without the loss of any accuracy after Algorithm \ref{alg:matrix2MPO} has finished. Alternatively, one can apply a rounding step on the intermediate result as soon as the MPO-rank reaches a certain threshold during the execution of the algorithm. The following example illustrates the necessity of the rounding step.
\begin{example}
Suppose we have three random $2\times 2$ matrices $\mat{A}^{(1)},\mat{A}^{(2)},\mat{A}^{(3)}$ such that
\begin{align*}
\mat{A} &= \mat{A}^{(3)} \otimes \mat{A}^{(2)} \otimes \mat{A}^{(1)}.
\end{align*}
By Theorem \ref{theo:unitMPT}, the matrix $\mat{A}$ has a canonical unit-rank MPO representation where the first MPO-tensor is $\mat{A}^{(1)}$ reshaped into a $1 \times 2 \times 2 \times 1$ tensor. Choosing $I_1=J_1=I_2=J_2=I_3=J_3=2$ and applying Algorithm \ref{alg:matrix2MPO} results in an MPO with a uniform rank of 16. Applying a rounding step truncates each of these ranks down to unity. 
\end{example}

For a fixed block size $I_1,J_1$ one can determine a lower bound for the resulting MPO-rank in the following manner.
\begin{lemma}
\label{lemma:lowerR}
Let $z$ be the number of nonzero elements of $\mat{A}$ and $I_1,J_1$ the first dimensions of the partitioning of $\mat{A}$. If $z=I_1\times J_1\times R$, then the minimal MPO-rank obtained by Algorithm \ref{alg:matrix2MPO} is $R$.
\end{lemma}
\begin{proof}
Suppose that we found a permutation such that all $z$ nonzero entries can be arranged into $R$ block matrices of size $I_1 \times J_1$. It then trivially follows that $R$ will be the MPO-rank since $z=I_1\times J_1 \times R$.
\end{proof}

In practice, it will be difficult, or in some cases impossible, to find a permutation such that all nonzero entries are nicely aligned into $I_1\times J_1$ block matrices. The $R$ in Lemma \ref{lemma:lowerR} is therefore a lower bound.

\subsection{Choosing a partition}
\label{subsec:partition}
In this subsection we discuss choosing a partition of the matrix $\mat{A}$ prior to applying Algorithm \ref{alg:matrix2MPO}. We suppose, without loss of generality, that the dimensions of $\mat{A}$ have prime factorizations $I=I_1\,I_2\cdots I_d$ and $J=J_1\,J_2 \cdots J_d$ with an equal amount of $d$ factors. The number of factors can always be made equal by appending ones. Ultimately, the goal is to obtain an MPO with ``small" MPO-ranks. Although Algorithm \ref{alg:matrix2MPO} constructs an MPO with ranks that are likely to exceed the canonical values, these ranks can always be truncated through rounding. Theorem \ref{theo:MPTranks} can be used for choosing a partition that minimizes the upper bounds in the hope that the canonical values are even smaller. The key idea is that the upper bounds depend on the ordering of the prime factors. The following small example illustrates.
\begin{example}
\label{ex:factorordering}
Suppose the factorizations are $I=35=7\times 5 \times 1$ and $J=12=3\times 2\times 2$. If we choose the ordering of the partition as $I_1=1,J_1=3,I_2=5,J_2=2,I_3=7,J_3=2$ then the upper bounds are $R_2 \leq \textrm{min}(3,140)=3$ and $R_3 \leq \textrm{min}(30,14)=14$. Choosing the partition $I_1=5,J_1=2,I_2=7,J_2=3,I_3=1,J_3=2$ changes the upper bounds to $R_2\leq 10$ and $R_3\leq 2$.
\end{example}

In light of the randomized SVD algorithm that is developed in Section \ref{sec:TNrSVD} it will be necessary to order the prime factors in a descending sequence. In this way, the first MPO-tensor will have sufficiently large dimension in order to compute the desired low-rank approximation. Observe that if we use the descending ordering $I_1=7,J_1=3,I_2=5,J_2=2,I_3=1,J_3=2$ in Example \ref{ex:factorordering}, then the upper bounds are $R_2 \leq \textrm{min}(21,20)=20$ and $R_3 \leq \textrm{min}(210,2)=2$. Large matrices can have dimensions with a large number of prime factors. Choosing a partition with a large number of MPO-tensors usually results in a high number of nonzero block matrices $\mat{X}$ and therefore also in a large MPO-rank. The problem with such a large MPO-rank can be that it becomes infeasible to do the rounding step due to lack of sufficient memory. In this case one needs to reduce the number of MPO-tensors until the obtained MPO-rank is small enough such that the rounding step can be performed. This way of choosing a partition will be demonstrated in more detail by means of a worked-out example in Section \ref{subsec:exp1}.

\section{Tensor network randomized SVD}
\label{sec:TNrSVD}
\subsection{The rSVD algorithm}
Given a matrix $\mat{A}$, which does not needs to be sparse, the rSVD computes a low-rank factorization $\mat{U}\mat{S}\mat{V}^T$ where $\mat{U},\mat{V}$ are orthogonal matrices and $\mat{S}$ is a diagonal and nonnegative matrix. The prototypical rSVD algorithm~\cite[p.~227]{Halko2011} is given as pseudocode in Algorithm \ref{alg:rSVD}. When a rank-$K/2$ approximation is desired, we compute a rank-$K$ approximation after which only the first $K/2$ singular values and vectors are retained. This is called oversampling and for more details on this topic we refer the reader to~\cite[p.~240]{Halko2011}.It has been shown that a slow decay of the singular values of $\mat{A}$ results in a larger approximation error. The power iteration $(\mat{A}\mat{A}^T)^q$ tries to alleviate this problem by increasing the decay of the singular values while retaining the same left singular vectors of $\mat{A}$. Common values for $q$ are 1 or 2. Note that the computation of $\mat{Y}$ is sensitive to round-off errors and additional orthogonalization steps are required. For a large matrix $\mat{A}$, it quickly becomes infeasible to compute orthogonal bases for $\mat{Y}$ or to compute the SVD of $\mat{B}$. This is the main motivation for doing all steps of Algorithm \ref{alg:rSVD} in an MPO-form. We therefore assume that all matrices in Algorithm \ref{alg:rSVD} can be represented by an MPO with relatively small MPO-ranks. For a matrix $\mat{A} \in \mathbb{R}^{I_1\cdots I_d \times J_1\cdots J_d}$ with $d$ MPO-tensors $\ten{A}^{(i)} \in \mathbb{R}^{R_i \times I_i \times J_i \times R_{i+1}} \, (i=1,\ldots,d)$, Algorithm \ref{alg:rSVD} computes a rank-$K$ factorization that consists of $d$ MPO-tensors $\ten{U}^{(1)},\ldots,\ten{U}^{(d)}$ and $\ten{V}^{(1)},\ldots,\ten{V}^{(d)}$ and the $K\times K$ diagonal and nonnegative $\mat{S}$ matrix. We denote this MPO-version of the rSVD algorithm the tensor network randomized SVD (TNrSVD).\\
\\
\framebox[.99\textwidth][l]
{\begin{minipage}{0.99\textwidth}
\begin{algorithm}Prototypical rSVD algorithm~\cite[p.~227]{Halko2011}\\
\textit{\textbf{Input}}: matrix $\mat{A} \in \mathbb{R}^{I \times J}$, target number $K$ and exponent $q$\\
\textit{\textbf{Output}}:\makebox[0pt][l]{ approximate rank-$K$ factorization $\mat{U}\mat{S}\mat{V}^T,$ where $\mat{U},\mat{V}$ are orthogonal}\\
\makebox[0pt][l]{\quad \quad \quad \quad and $\mat{S}$ is diagonal and nonnegative.}
\begin{algorithmic}
\STATE Generate an $J \times K$ random matrix $\mat{O}$.
\STATE $\mat{Y} \gets (\mat{A}\mat{A}^T)^q\,\mat{A}\,\mat{O}$
\STATE $\mat{Q} \gets$ Orthogonal basis for the range of $\mat{Y}$
\STATE $\mat{B} \gets \mat{Q}^T\,\mat{A}$
\STATE Compute the SVD $\mat{B}=\mat{W}\mat{S}\mat{V}^T$.
\STATE $\mat{U} \gets \mat{Q}\,\mat{W}$
\end{algorithmic}
\label{alg:rSVD}
\end{algorithm}
\end{minipage}}\\
\\
\\
The rSVD algorithm relies on multiplying the original matrix $\mat{A}$ with a random matrix $\mat{O}$. Fortunately, it is possible to directly construct a random matrix into MPO form.
\begin{lemma}
\label{theo:randomMPO}
A particular random $J_1J_2\cdots J_d  \times K$ matrix $\mat{O}$ with $\textrm{rank}(\mat{O})=K$ and $K \leq J_1$ can be represented by a unit-rank MPO with the following random MPO-tensors
\begin{align*}
\ten{O}^{(1)} \in \mathbb{R}^{1 \times J_1 \times K \times 1},\\
\ten{O}^{(i)} \in \mathbb{R}^{1 \times J_i \times 1 \times 1}, \, (i=2,\ldots,d).
\end{align*}
\end{lemma}
\begin{proof}
All MPO-ranks being equal to one implies that Theorem \ref{theo:unitMPT} applies. The random matrix $\mat{O}$ is constructed from the Kronecker product of $d-1$ random column vectors $\mat{O}^{(i)} \in \mathbb{R}^{J_i},\, (i=2,\ldots,d)$ with the matrix $\mat{O}^{(1)} \in \mathbb{R}^{J_1 \times K}$. The Kronecker product has the property that
\begin{align*}
\textrm{rank}(\mat{O}) &= \textrm{rank}(\mat{O}^{(d)}) \, \cdots \textrm{rank}(\mat{O}^{(1)}).
\end{align*}
The fact that $\mat{O}^{(1)}$ is a random matrix then ensures that $\textrm{rank}(\mat{O})=K$.
\end{proof}

Probabilistic error bounds for Algorithm \ref{alg:rSVD} are typically performed for random Gaussian matrices $\mat{O}$~\cite[p.~273]{Halko2011}. Another 
type of test matrices $\mat{O}$ are subsampled random Fourier transform matrices~\cite[p.~277]{Halko2011}. The random matrix in MPO form from Theorem \ref{theo:randomMPO} will not be Gaussian, as the multiplication of Gaussian random variables is not Gaussian. This prevents the straightforward determination of error bounds for the MPO-implementation of Algorithm \ref{alg:rSVD} that we propose. In spite of the lack of any probabilistic bounds on the error, all numerical experiments that we performed demonstrate that the orthogonal basis that we obtain for the range of $\mat{A}$ can capture the action of $\mat{A}$ sufficiently. Once the matrix $\mat{A}$ has been converted into an MPO using Algorithm \ref{alg:matrix2MPO} and a random MPO has been constructed using Theorem \ref{theo:randomMPO}, what remains are matrix multiplications and computing low-rank QR and SVD factorizations. We will now explain how these steps can be done efficiently using MPOs.
\subsection{Matrix multiplication}
Matrix multiplication is quite straightforward. Suppose the matrices $\mat{A} \in \mathbb{R}^{I_1I_2\cdots I_d \times J_1J_2\cdots J_d}, \mat{O} \in \mathbb{R}^{J_1J_2\cdots J_d \times K}$ have MPO representations of 4 tensors. This implies that the rows and columns of $\mat{A}$ are indexed by the multi-indices $[i_1i_2i_3i_4],[j_1j_2j_3j_4]$, respectively. The matrix multiplication $\mat{A}\,\mat{O}$ then corresponds with the summation of the column indices of $\mat{A}$
\begin{align*}
\mat{A}\,\mat{O} &= \sum_{j_1,j_2,j_3,j_4}\, \mat{A}(:,[j_1j_2j_3j_4])\,\mat{O}([j_1j_2j_3j_4],:)
\end{align*}
and is visualized as contractions of two MPOs meshed into one tensor network in Figure~\ref{fig:AO}, where unlabelled indices have a dimension of one. The contraction $\sum_{j_i}\ten{A}^{(i)}(:,j_i,:)\ten{O}^{(i)}(:,j_i,:)$ for each of the four MPO-tensors results in a new MPO that represents the matrix multiplication $\mat{A}\,\mat{O}$. If $\ten{A}^{(i)} \in \mathbb{R}^{R_i \times I_i \times J_i \times R_{i+1}}$ and $\ten{O}^{(i)} \in \mathbb{R}^{S_i \times J_i \times 1 \times S_{i+1}}$, then the summation over the index $J_i$ results in an MPO-tensor with dimensions $R_iS_i \times I_i \times 1 \times R_{i+1}S_{i+1}$ with a computational complexity of $O(R_iS_iI_iJ_iR_{i+1}S_{i+1})$ flops. Corresponding MPO-ranks $R_i,S_i$ and $R_{i+1},S_{i+1}$ are multiplied with one another, which necessitates a rounding step in order to reduce the dimensions of the resulting MPO-tensors. Note, however, that the random matrix $\mat{O}$ constructed via Lemma \ref{theo:randomMPO} has a unit-rank MPO, which implies that $S_1=S_2=\cdots=S_{d+1}=1$ such that the MPO corresponding with the matrix $\mat{A}\mat{O}$ will retain the MPO-ranks of $\mat{A}$.
\begin{figure}[h]
\begin{center}
\input{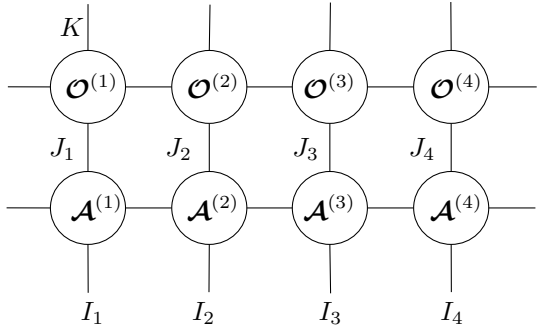}
\end{center}
\caption{The matrix multiplication $\mat{A}\,\mat{O}$ as contractions of a tensor network.}
\label{fig:AO}
\end{figure}

\subsection{Thin QR and economical SVD in MPO-form}
An orthogonal basis for the range of $\mat{Y} \in \mathbb{R}^{I \times K}$ can be computed through a thin QR decomposition $\mat{Y}=\mat{Q}\,\mat{R}$, where $\mat{Q} \in \mathbb{R}^{I \times K}$ has orthogonal columns and $\mat{R} \in \mathbb{R}^{K \times K}$. The algorithm to compute a thin QR decomposition from a matrix in MPO-form is given in pseudo code in Algorithm \ref{alg:MPOqr}. The thin QR is computed by an orthogonalization sweep from right-to-left, absorbing the $R$ factor matrix into the preceding MPO-tensor. The main operations in the orthogonalization sweep are tensor reshaping and the matrix QR decomposition. The first MPO-tensor is orthogonalized in a slightly different way such that the $K\times K$ $\mat{R}$ matrix is obtained. The computational cost of Algorithm \ref{alg:MPOqr} is dominated by the QR computation of the first MPO-tensor, as it normally has the largest dimensions. Using Householder transformations to compute this QR decomposition costs approximately $O(I_1^2R_2^2K)$ flops. The proof of the procedure can be found in~\cite[p.~2302]{ivanTT}. The Thin QR decomposition in MPO-form is illustrated for an MPO of 4 tensors in Figure \ref{fig:QR}. Again, all unlabeled indices have a dimension of one.\\
\\
\framebox[.99\textwidth][l]
{\begin{minipage}{0.99\textwidth}
\begin{algorithm}MPO-QR algorithm~\cite[p.~2302]{ivanTT}\\
\textit{\textbf{Input}}: rank-$K$ matrix $\mat{A} \in \mathbb{R}^{I \times K}$ in MPO-form with $\ten{A}^{(1)} \in \mathbb{R}^{1 \times I_1 \times K \times R_2}$, $I_1 \geq K$.\\
\textit{\textbf{Output}}:\makebox[0pt][l]{ $d$ MPO-tensors of $\mat{Q} \in \mathbb{R}^{I\times K}$ with $\ten{Q}^{(1)} \in \mathbb{R}^{1 \times I_1 \times K \times R_2}$, $\mat{Q}^T\mat{Q}=\mat{I}_K$}\\
\makebox[0pt][l]{\quad \quad \quad \quad and $\mat{R} \in \mathbb{R}^{K \times K}.$}
\begin{algorithmic}
\FOR {i=d:-1:2}
\STATE Reshape $\ten{A}^{(i)}$ into $R_i \times I_iR_{i+1}$ matrix $\mat{A}_i$.
\STATE $\mat{A}_i= \mat{R}_i\,\mat{Q}_i$ with $\mat{R}_i \in \mathbb{R}^{R_i \times R_i}$ and $\mat{Q}_i\mat{Q}_i^T=\mat{I}_{R_i}$.
\STATE $\ten{Q}^{(i)} \gets $ reshape $\mat{Q}_i$ into $R_i \times I_i \times 1 \times R_{i+1}$ tensor.
\STATE $\ten{A}^{(i-1)} \gets \ten{A}^{(i-1)} \times_4 R_i$.
\ENDFOR
\STATE Permute $\ten{A}^{(1)}$ into $K \times 1 \times I_1\times R_2$ tensor.
\STATE Reshape $\ten{A}^{(1)}$ into $K \times I_1R_2$ matrix $\mat{A}_1$.
\STATE $\mat{A}_1= \mat{R}\,\mat{Q}_1$ with $\mat{R} \in \mathbb{R}^{K \times K}$ and $\mat{Q}_1\mat{Q}_1^T=\mat{I}_{K}$.
\STATE Reshape $\mat{Q}_1$ into $K \times 1 \times I_1\times R_2$ tensor $\ten{Q}^{(1)}$.
\STATE Permute $\ten{Q}^{(1)}$ into $1 \times I_1 \times K \times R_2$ tensor. 
\end{algorithmic}
\label{alg:MPOqr}
\end{algorithm}
\end{minipage}}\\
\\
\\
\begin{figure}[h]
\begin{center}
\input{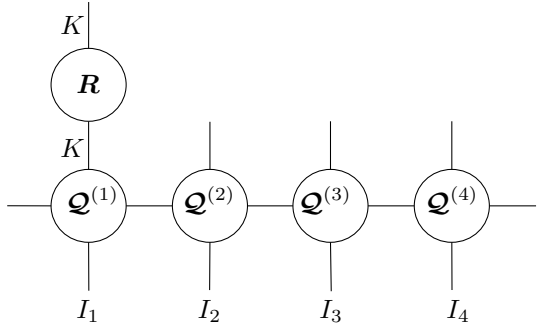}
\end{center}
\caption{Thin QR decomposition $\mat{Q}\mat{R}$ as a tensor network with $\mat{Q}\in \mathbb{R}^{I_1I_2I_3I_4\times K}$ and $\mat{R}\in \mathbb{R}^{K\times K}$.}
\label{fig:QR}
\end{figure}

The rSVD algorithm also requires an economical SVD computation of the $K\times J_1\cdots J_d$ matrix $\mat{B}=\mat{W}\,\mat{S}\,\mat{V}^T$, where both $\mat{W}, \mat{S}$ are $K\times K$ matrices, $\mat{W}$ is orthogonal and $\mat{S}$ is diagonal and nonnegative. The matrix $\mat{V}$ is stored in MPO-form. Only a slight modification of Algorithm \ref{alg:MPOqr} is required to obtain the desired matrices. Indeed, the only difference with Algorithm \ref{alg:MPOqr} is that now the SVD of $\mat{A}_1$ needs to be computed. From this SVD we obtain the desired $\mat{W},\mat{S}$ matrices and can reshape and permute the right singular vectors into the desired $\ten{V}^{(1)}$ MPO-tensor. Again, the overall computational cost will be dominated by this SVD step, which costs approximately $O(J_1^2R_2^2K)$ flops. The economical SVD of a matrix in MPO-form is given in pseudo-code in Algorithm \ref{alg:MPOsvd}. A graphical representation of the corresponding tensor network for a simple example of 4 MPO-tensors is depicted in Figure~\ref{fig:SVD}. \\
\\
\framebox[.99\textwidth][l]
{\begin{minipage}{0.99\textwidth}
\begin{algorithm}MPO-SVD algorithm\\
\textit{\textbf{Input}}: $d$ MPO-tensors of $\mat{A} \in \mathbb{R}^{K \times J}$ with $\ten{A}^{(1)} \in \mathbb{R}^{1 \times K \times J_1 \times R_2}$ and $J_1 \geq K$.\\
\textit{\textbf{Output}}:\makebox[0pt][l]{ $d$ MPO-tensors of $\mat{V} \in \mathbb{R}^{K\times J}$ with $\ten{V}^{(1)} \in \mathbb{R}^{1 \times K \times J_1 \times R_2}$, $\mat{V}\mat{V}^T=\mat{I}_K$}\\
\makebox[0pt][l]{\quad \quad \quad \quad and $\mat{W} \in \mathbb{R}^{K \times K}$, $\mat{W}^T\mat{W}=\mat{I}$, and $\mat{S} \in \mathbb{R}^{K \times K}$ diagonal and nonnegative.}  
\begin{algorithmic}
\FOR {i=d:-1:2}
\STATE Reshape $\ten{A}^{(i)}$ into $R_i \times I_iR_{i+1}$ matrix $\mat{A}_i$.
\STATE $\mat{A}_i= \mat{R}_i\,\mat{Q}_i$ with $\mat{R}_i \in \mathbb{R}^{R_i \times R_i}$ and $\mat{Q}_i\mat{Q}_i^T=\mat{I}_{R_i}$.
\STATE $\ten{V}^{(i)} \gets $ reshape $\mat{Q}_i$ into $R_i \times I_i \times 1 \times R_{i+1}$ tensor.
\STATE $\ten{A}^{(i-1)} \gets \ten{A}^{(i-1)} \times_4 R_i$.
\ENDFOR
\STATE Permute $\ten{A}^{(1)}$ into $K \times 1 \times J_1\times R_2$ tensor.
\STATE Reshape $\ten{A}^{(1)}$ into $K \times J_1R_2$ matrix $\mat{A}_1$.
\STATE Compute SVD of $\mat{A}_1= \mat{W}\,\mat{S}\,\mat{Q}_1^T$.
\STATE Reshape $\mat{Q}_1$ into $K \times 1 \times J_1\times R_2$ tensor $\ten{V}^{(1)}$.
\STATE Permute $\ten{V}^{(1)}$ into $1 \times K \times J_1 \times R_2$ tensor. 
\end{algorithmic}
\label{alg:MPOsvd}
\end{algorithm}
\end{minipage}}\\
\\
\\
\begin{figure}[h]
\begin{center}
\input{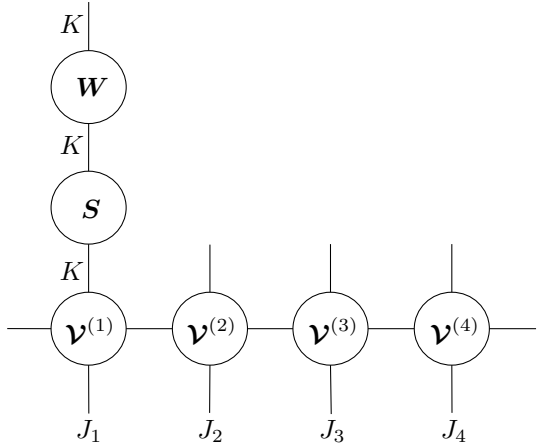}
\end{center}
\caption{Economical SVD $\mat{W}\mat{S}\,\mat{V}^T$ as a tensor network with $\mat{V}\in \mathbb{R}^{J_1J_2J_3J_4 \times K}$ and $\mat{W},\mat{S} \in \mathbb{R}^{K\times K}$.}
\label{fig:SVD}
\end{figure}
Both the thin QR and economical SVD of the matrix in MPO-form are computed for each tensor of the MPO separately, which reduces the computational complexity significantly. Unlike the ALS-SVD and MALS-SVD, no iterative sweeping over the different MPO-tensors is required.

\subsection{Randomized subspace iteration}
The computation of the matrix~$\mat{Y}=(\mat{A}\mat{A}^T)^q\,\mat{A}\mat{O}$ is vulnerable to round-off errors and an additional orthogonalization step is required between each application of $\mat{A}$ and $\mat{A}^T$~\cite[p.~227]{Halko2011}. Instead of computing $\mat{Y}$ and applying Algorithm \ref{alg:MPOqr}, the randomized MPO-subspace iteration of Algorithm \ref{alg:rsubspaceiter} is proposed. First, the random matrix $\mat{O}$ is multiplied onto $\mat{A}$, after which a rounding step is performed to reduce the MPO-ranks. Algorithm \ref{alg:MPOqr} is then applied to obtain an orthogonal basis $\mat{Q}$ for the range of $\mat{Y}$. One now proceeds with the multiplication $\mat{A}^T\mat{Q}$, after which another rounding step and orthogonalization through Algorithm \ref{alg:MPOqr} are performed. These steps are repeated until the desired number of multiplications with $\mat{A}$ and $\mat{A}^T$ have been done. The SVD-based rounding and orthogonalization steps can actually be integrated into one another. Indeed, one can apply a left-to-right rounding sweep first, followed by the right-to-left sweep of Algorithm \ref{alg:MPOqr}. This prevents performing the right-to-left sweep twice. Similarly, one can integrate the rounding step after the multiplication $\mat{Q}^T\mat{A}$ with the computation of the economical SVD $\mat{W}\mat{S}\mat{V}^T$. Also note that since $\mat{W} \in \mathbb{R}^{K \times K}$, the multiplication $\mat{Q}\mat{W}$ in MPO-form is equivalent with $\ten{Q}^{(1)} \times_3 \mat{W}^T$.\\
\\
\framebox[.99\textwidth][l]
{\begin{minipage}{0.99\textwidth}
\begin{algorithm}Randomized MPO-subspace iteration\\
\textit{\textbf{Input}}: $\mat{A} \in \mathbb{R}^{I \times J}$ and random matrix $\mat{O} \in \mathbb{R}^{J \times K}$ in MPO-form.\\
\textit{\textbf{Output}}:\makebox[0pt][l]{ MPO-tensors of $\mat{Q} \in \mathbb{R}^{I \times K}$ with $\mat{Q}^T\mat{Q}=\mat{I}_K$.}
\begin{algorithmic}
\STATE $\mat{Y} \gets \mat{A}\mat{O}$
\STATE $\mat{Q} \gets$ Use Algorithm \ref{alg:MPOqr} on $\mat{Y}$\FOR {i=1:q}
\STATE $\mat{Y} \gets \mat{A}^T\,\mat{Q}$
\STATE $\mat{Q} \gets$ Use Algorithm \ref{alg:MPOqr} on $\mat{Y}$
\STATE $\mat{Y} \gets \mat{A}\,\mat{Q}$
\STATE $\mat{Q} \gets$ Use Algorithm \ref{alg:MPOqr} on $\mat{Y}$
\ENDFOR
\end{algorithmic}
\label{alg:rsubspaceiter}
\end{algorithm}
\end{minipage}}\\
\\
\\

\section{Numerical Experiments}
\label{sec:experiments}
In this section we demonstrate the effectiveness of the algorithms discussed in this article. Algorithms \ref{alg:matrix2MPO} up to \ref{alg:rsubspaceiter} were implemented in MATLAB and run on a desktop computer with an 8-core Intel i7-6700 cpu @ 3.4 GHz and 64 GB RAM. These implementations can be freely downloaded from \url{https://github.com/kbatseli/TNrSVD}. 

\subsection{Matrix permutation prior to MPO conversion}
\label{subsec:exp1}
Applying a permutation prior to the conversion can effectively reduce the maximal MPO-rank. Consider the $150102 \times 150102$ AMD-G2-circuit matrix from the UF Sparse Matrix Collection~\cite{Davis2011}, with a bandwidth of 93719 and sparsity pattern shown in Figure \ref{fig:A}. The sparsity pattern after applying the Cuthill-Mckee algorithm is shown in Figure~\ref{fig:Ap} and the bandwidth is reduced to 1962. We can factor 150102 as $2\times 3\times 3\times 31\times 269$, which sets the maximal number of tensors in the MPO to 6. The total number of nonzero entries is 726674, which makes an MPO representation of 6 tensors infeasible as for this case all $R_k=726674$ and there is insufficient memory to store the MPO-tensors. Table~\ref{table:maxR} lists the number of MPO-cores $d$, the obtained MPO-rank for both the original matrix and after applying the Cuthill-Mckee algorithm and the runtime for applying Algorithm \ref{alg:matrix2MPO} on the permuted matrix. Applying the permutation effectively reduces the MPO-rank approximately by half so we only consider the permuted matrix. First, we order the prime factors in a descending fashion $269,31,3,3,2$, which would result in an MPO that consists of the following five tensors $\ten{A}^{(1)}\in \mathbb{R}^{1\times 269\times 269 \times 4382}, \ten{A}^{(2)} \in \mathbb{R}^{4382\times 31 \times 31 \times 4382}, \ten{A}^{(3)} \in \mathbb{R}^{4382\times 3 \times 3 \times 4382},\ten{A}^{(4)} \in \mathbb{R}^{4382 \times 3 \times 3 \times 4381},\ten{A}^{(5)}\in \mathbb{R}^{4382 \times 2 \times 2 \times 4382}$. Due to the high MPO-rank, however, it is not possible to construct this MPO. We can now try to absorb the prime factor 2 into 269 and construct the corresponding MPO that consists of four tensors with $\ten{A}^{(1)}\in \mathbb{R}^{1\times 538 \times 538 \times 1753}$. This takes about 5 seconds. As Table \ref{table:maxR} shows, increasing the dimensions of $\ten{A}^{(1)}$ by absorbing it with more prime factors further reduces the MPO-rank and runtimes. Note that if MATLAB supported sparse tensors by default, then it would be possible to use Algorithm \ref{alg:matrix2MPO} for both the original and permuted matrix as all MPO-tensors are sparse. SVD-based rounding on the $d=3$ MPO with a tolerance of $10^{-10}$ reduces the MPO-rank from 347 down to $7$ and takes about 61 seconds. Using the alternative parallel vector rounding from \cite{Hubig2017} truncates the MPO-rank also down to 7 in about 5 seconds.
\begin{figure}
\centering
\begin{minipage}{.5\textwidth}
  \centering
  \includegraphics[height=6cm]{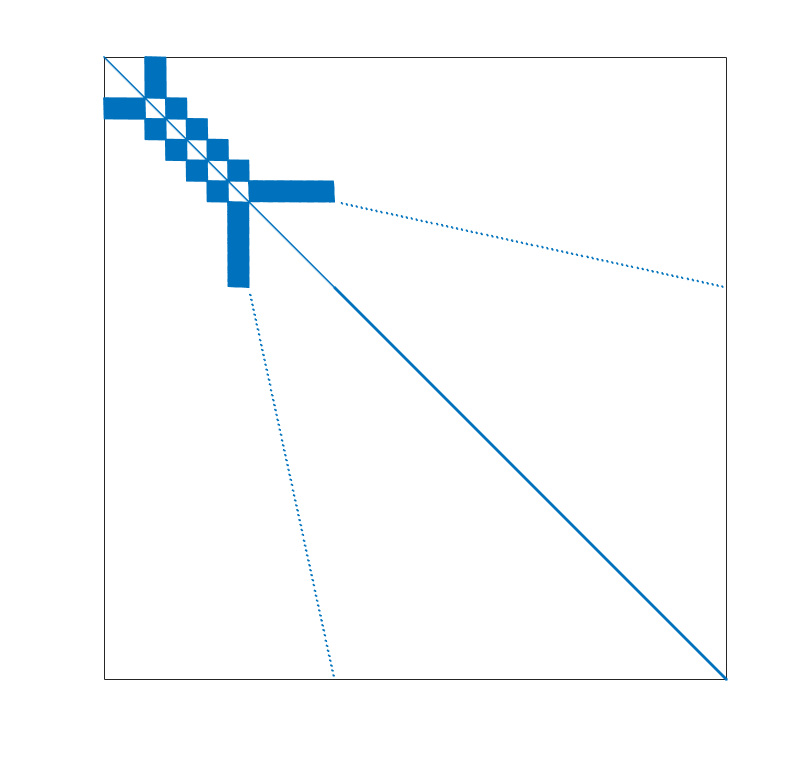}
  \captionof{figure}{Original matrix\\sparsity pattern.}
  \label{fig:A}
\end{minipage}%
\begin{minipage}{.5\textwidth}
  \centering
  \includegraphics[height=6cm]{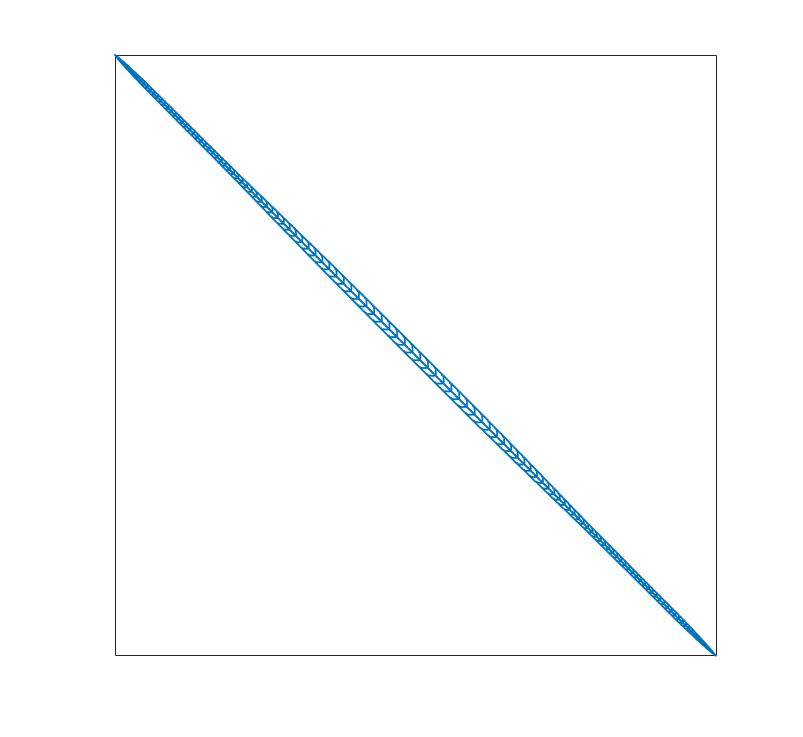}
  \captionof{figure}{Sparsity pattern after the\\Cuthill-Mckee permutation.}
  \label{fig:Ap}
\end{minipage}
\end{figure}

\begin{table}[htbp]
\newcommand{\tabincell}[2]{\begin{tabular}{@{}#1@{}}#2\end{tabular}}
\centering
\caption{\label{table:maxR}Maximal MPO-ranks for varying block matrix sizes.}
\vspace{1ex}
\begin{tabular}{|c|c|c|c|c|}
 \hline
 \multirow{2}{*}{Matrix block size} & \multirow{2}{*}{$d$} &
 \multicolumn{2}{c|}{Maximal Rank} &  \multirow{2}{*}{Runtime} \\ [1.5ex]
 \cline{3-4} & 
   & Original & Permuted  & [seconds]\\
 \hline
$269\times 269$ & 5 & 9352 & 4382 & NA \\
$538\times 538$ & 4 & 3039 & 1753 & 5.47\\
$807\times 807$ & 4 & 1686 & 1018 & 4.67\\
$1614\times 1614$& 3  & 665 & 347 & 3.92\\
\hline
 \end{tabular}
\end{table}

\subsection{Fast matrix-to-MPO conversion}
In this experiment, Algorithm \ref{alg:matrix2MPO} is compared with the TT-SVD algorithm~\cite[p.~2135]{Oseledets2010} and the TT-cross algorithm~\cite[p.~82]{ttcross},~\cite{Savostyanov2011}, both state-of-the-art methods for converting a matrix into an MPO. Both the TT-SVD and TT-cross implementations from the TT-Toolbox\cite{TTToolbox} were used. The TT-cross method was tested with the DMRG\_CROSS function and was run for 10 sweeps and with an accuracy of $10^{-10}$ We use the $15838 \times 15838$ power simulation matrix from Liu Wenzhuo, EPRI, China, also in the UF Sparse Matrix Collection~\cite{Davis2011}. The prime factorization of 15838 consists of only two prime factors, 2 and 7919, and is not very suitable for a MPO conversion. We therefore append the matrix with zeros such that its dimensions are rounded to the nearest power of 2, as the SVD of the original matrix is easily recovered from the appended matrix. An MPO of 7 tensors is then constructed from the appended $16384 \times 16384$ matrix with dimensions $I_1=J_1=256,I_2=\cdots=I_7=J_1=\cdots=J_7=2$. Applying the Cuthill-Mckee algorithm reduces the maximal MPO-rank from 413 to 311 and is therefore not applied as the resulting rank decrease is not very significant. Table \ref{table:matrix2mpo} lists the runtimes and relative errors when converting the appended matrix into the MPO format for the three considered methods. The relative errors are obtained by contracting the obtained MPO back into a matrix $\hat{\mat{A}}$ and computing $||\mat{A}-\hat{\mat{A}}||_F/||\mat{A}||_F$. Both Algorithm \ref{alg:matrix2MPO} and the TT-SVD manage to obtain a result that is accurate up to machine precision. Although the maximal MPO-rank obtained with the TT-SVD algorithm is 22, which is an order of magnitude smaller than 413, our proposed algorithm is about 509 times faster than the TT-SVD algorithm. The TT-cross method fails to find a sufficiently accurate MPO and takes about the same amount of time as the TT-SVD algorithm. Applying the TT-SVD and TT-cross method on the AMD-G2-circuit matrix was not possible due to insufficient memory.
\begin{table}[tb]
\begin{center}
\caption{\label{table:matrix2mpo}Runtimes and relative errors for three different matrix to MPO methods.}
\begin{tabular}{@{}lccc@{}}
 		 &  Algorithm \ref{alg:matrix2MPO} & TT-SVD & TT-cross \\
          \midrule
Runtime [s]		 &  0.1396 & 71.035  & 81.746 \\
Relative error   &  9.4e-15& 4.8e-15  & 0.65 \\

\end{tabular}
\end{center}
\end{table}

\subsection{Comparison with ALS-SVD and MALS-SVD}
The tensor network-method described in~\cite{Lee2015} uses the alternating least squares (ALS) and modified alternating least squares (MALS) methods to compute low-rank approximations of a given matrix in MPO-form. Three numerical experiments are considered in~\cite{Lee2015}, of which two deal with finding a low-rank approximation to a given matrix. The first matrix that is considered is a rectangular submatrix of the Hilbert matrix. The Hilbert matrix $\mat{H} \in \mathbb{R}^{2^N \times 2^N}$ is a symmetric matrix with entries $\mat{H}(i,j)=(i+j-1)^{-1}, i,j=1,2,\ldots,2^N$. For this experiment, the submatrix $\mat{A} := \mat{H}(:,1:2^{N-1})$ is considered with $10\leq N \leq 50$. Following~\cite{Lee2015}, the corresponding MPO is constructed using the FUNCRS2 function of the TT-Toolbox~\cite{TTToolbox}, which applies a TT-cross approximation method using a functional description of the matrix entries, and consists of $N$ MPO-tensors. The obtained MPO approximates the Hilbert matrix with a relative error of $10^{-11}$. Maximal MPO-ranks for all values of $N$ were bounded between 18 and 24. A tolerance for the relative residual of $10^{-8}$ was set for the computation of rank-16 approximations with both the ALS and MALS algorithms. In order to be able to apply Algorithm \ref{alg:rSVD}, we first need to make sure that the MPO-tensor $\ten{Y}^{(1)}$ of the matrix $\mat{Y}=\mat{A}\mat{O}$ has dimensions $1 \times I_1 \times K \times R_2$ with $I_1 \geq K$. Since a rank-16 approximation is desired, this means that $K=32$. By contracting the first 5 MPO-tensors $\ten{A}^{(1)},\ten{A}^{(2)},\ten{A}^{(3)},\ten{A}^{(4)},\ten{A}^{(5)}$ into a tensor with dimensions $1 \times 32 \times 32 \times R_6$, we obtain a new MPO of $N-5+1$ tensors that satisfies the $I_1 \geq K$ condition. A tolerance of $10^{-9}$ was used in the rounding procedure and $q$ was set to $2,3,4$ for $N=10,\ldots,30$, $N=35,40$ and $N=45,50$, respectively. Figure~\ref{fig:hilbert} shows the runtimes of the ALS, MALS and TNrSVD method as a function of $N$. All computed rank-16 approximations obtained from the ALS and MALS methods satisfied $||\mat{A}-\mat{U}\mat{S}\mat{V}^T||_F / ||\mat{A}||_F  \leq 10^{-8}$, as reported by the respective methods. As mentioned in \cite{Lee2015}, computing the norm of the residual in MPO-form can be computationally challenging when the MPO-ranks are high. For this reason, we compared the obtained singular values from the ALS method with the singular values from the TNrSVD method to ensure the relative residuals were below $10^{-8}$. The MALS method solves larger optimization problems than the ALS method at each iteration and is therefore considerably slower. The TNrSVD method is up to 6 times faster than the ALS method and 13 times faster than the MALS method for this particular example. Using a standard matrix implementation of Algorithm \ref{alg:rSVD} we could compute low-rank approximations only for the $N=10$ and $N=15$ cases, with respective runtimes of $0.04$ and $9.79$ seconds. From this we can conclude that all three tensor-based methods outperform the standard matrix implementation of Algorithm \ref{alg:rSVD} when $N \geq 15$ for this particular example.
\begin{figure}[h]
\begin{center}
\includegraphics[width=.9\textwidth]{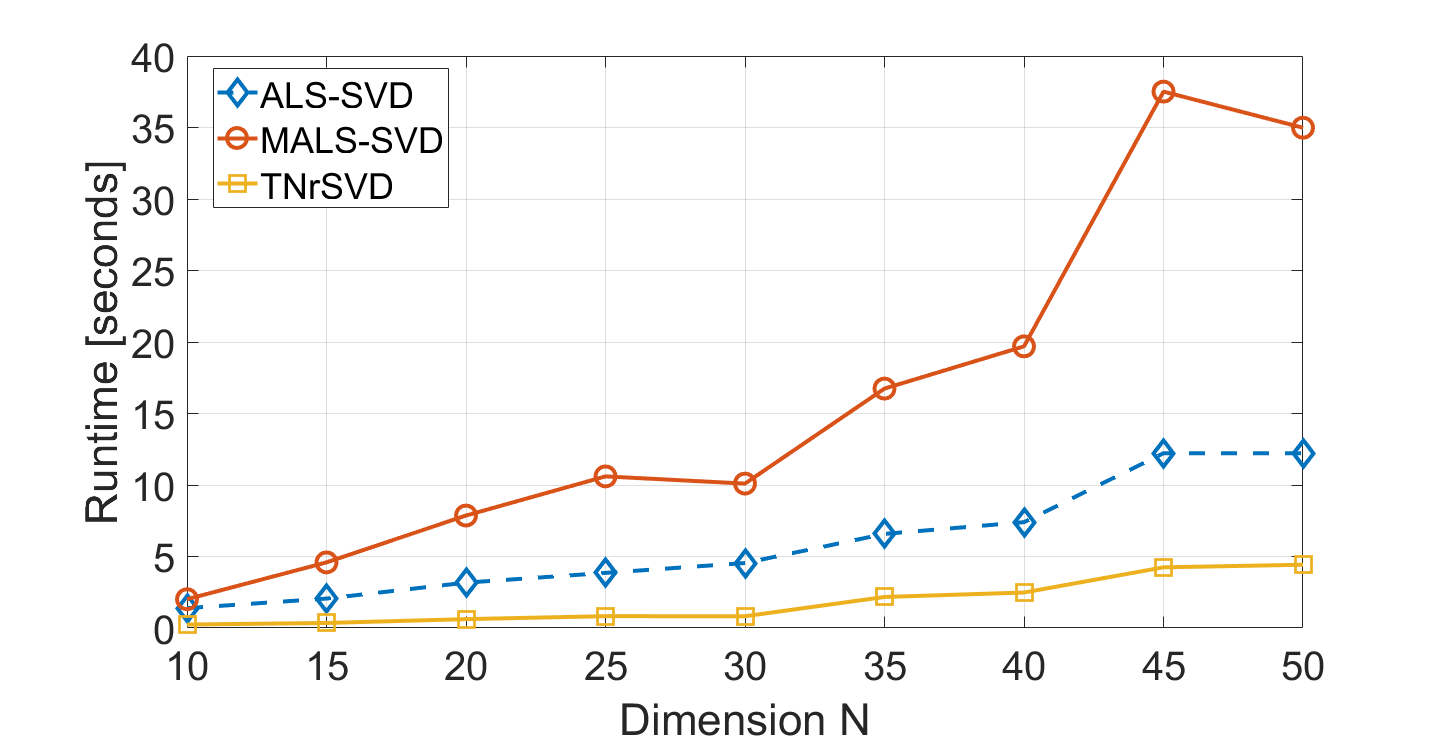}
\end{center}
\caption{Runtimes for computing rank-16 approximations of $2^N \times 2^{N-1}$ Hilbert matrices for $10\leq N \leq 50$.}
\label{fig:hilbert}
\end{figure}

The second matrix that is considered is one with 50 prescribed singular values $0.5^k,\,(k=0,\ldots,49)$ and random left and right singular vectors $\mat{U} \in \mathbb{R}^{2^N \times 50},\mat{V}\in \mathbb{R}^{2^N \times 50}$. As with the Hilbert matrices, $N$ ranges from 10 up to 50 and equals the number of tensors in the MPO. The maximal MPO-rank of all constructed MPOs was 25. The orthogonal $\mat{U},\mat{V}$ matrices were generated in MPO-form using the TT\_RAND function of the TT-Toolbox. The MPO-representation of the matrix was then obtained by computing $\mat{U}\mat{S}\mat{V}^T$ in MPO-form. A tolerance of $10^{-6}$ was set for the computation of rank-50 approximations with both the ALS and MALS algorithms such that all approximations satisfy $||\mat{S}-\mat{\hat{S}}||_F / ||\mat{S}||_F \leq 10^{-6}$, where $\mat{\hat{S}}$ denotes the diagonal matrix obtained from either the ALS or MALS method. One sweep of the ALS and MALS algorithms sufficed to obtain the desired accuracy, except for the cases $N=45$ and $N=50$. For these two cases neither ALS nor MALS was able to converge to the desired accuracy. The exponent $q$ was set to 1 for all runs of the TNrSVD algorithm and the rounding tolerances were set to $10^{-5},10^{-6},10^{-8},10^{-8},10^{-9},10^{-10},10^{-11},10^{-12},10^{-13}$ for $N=10,15,20,\ldots,50$, respectively. This ensured that the result of the TNrSVD method had a relative error $||\mat{S}-\mat{\hat{S}}||_F / ||\mat{S}||_F$ on the estimated $50$ dominant singular values below $10^{-6}$. Computing a rank-50 approximation implies that $K=100$ and the first 7 MPO tensors need to be contracted prior to running the TNrSVD algorithm. These contractions result in a new MPO where the first tensor has dimensions $1 \times 128 \times 128 \times R_8$, such that $I_1=128 \geq K=100$ is satisfied. Figure~\ref{fig:lowranksvd} shows the runtimes of the ALS, MALS and TNrSVD method as a function of $N$. Just like with the Hilbert matrices, the MALS algorithm takes considerately longer to finish one sweep. The TNrSVD algorith is up to 5 times faster than ALS and 17 times faster than MALS for this particular example. Using a standard matrix implementation of Algorithm \ref{alg:rSVD} we could compute low-rank approximations only for the $N=10$ and $N=15$ cases, with respective runtimes of $0.04$ and $7.59$ seconds. For $N=15$, the runtimes for the ALS, MALS and TNrSVD methods were $7.63,19.5$ and $2.13$ seconds, respectively.
\begin{figure}[h]
\begin{center}
\includegraphics[width=.9\textwidth]{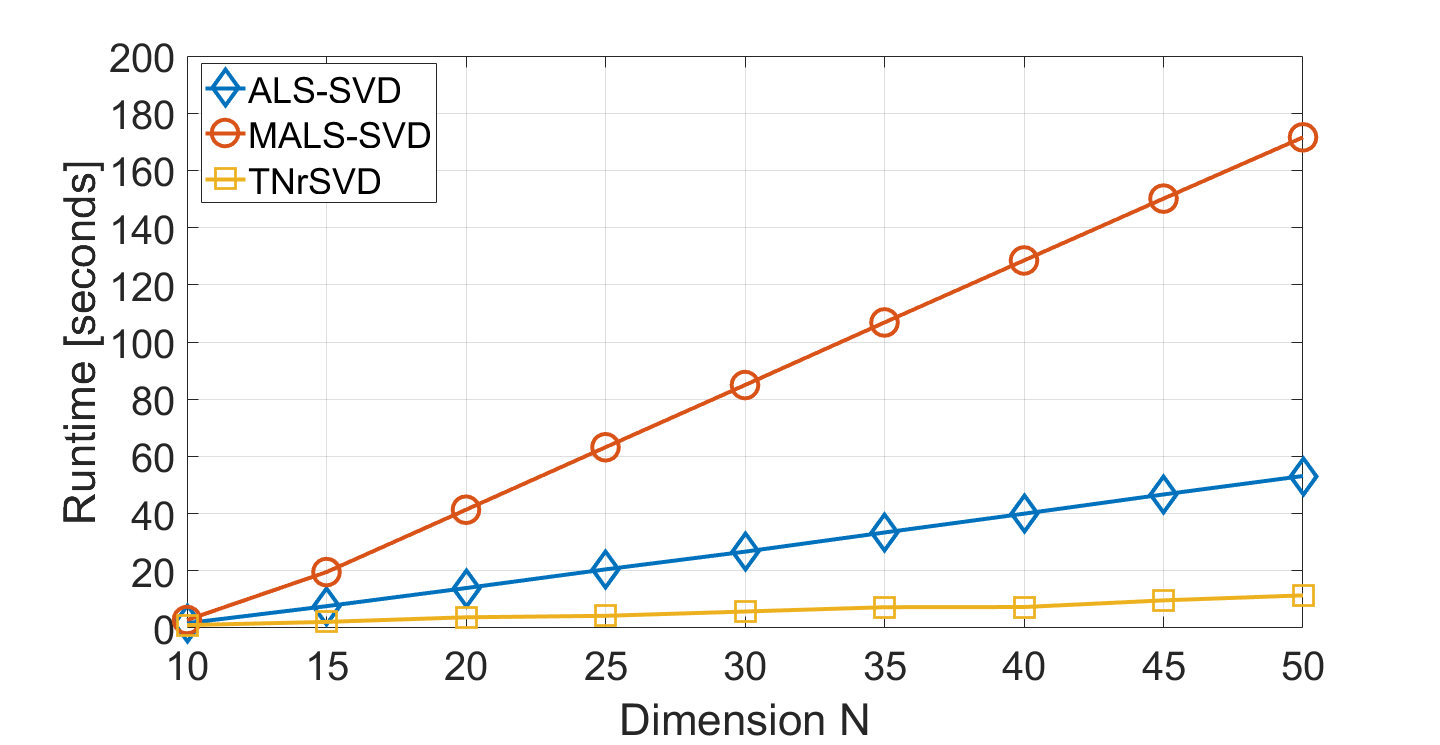}
\end{center}
\caption{Runtimes for computing rank-50 decompositions of $2^N \times 2^{N}$ random matrices with prescribed singular values for $10\leq N \leq 50$.}
\label{fig:lowranksvd}
\end{figure}

\section{Conclusion}
\label{sec:conclusions}
We have proposed a new algorithm to convert a sparse matrix into an MPO form and a new randomized algorithm to compute a low-rank approximation of a matrix in MPO form. Our matrix to MPO conversion algorithm is able to generate MPO representations of a given matrix with machine precision accuracy up to 509 times faster than the standard TT-SVD algorithm. Compared with the state-of-the-art ALS-SVD and MALS-SVD algorithms, our TNrSVD is able to find low-rank approximations with the same accuracy up to 6 and 17 times faster, respectively. Future work includes the investigation of finding permutations, other than the Cuthill-Mckee permutation, that can reduce the MPO-rank of a given matrix. 


 \bibliographystyle{siam}
 \bibliography{references.bib}

\end{document}